\documentclass[12pt]{amsart}

\usepackage{amssymb}

\usepackage{enumerate}

\usepackage{graphicx}

\makeatletter
\@namedef{subjclassname@2010}{%
  \textup{2010} Mathematics Subject Classification}
\makeatother

\newtheorem{thm}{Theorem}[section]
\newtheorem{cor}[thm]{Corollary}
\newtheorem{lem}[thm]{Lemma}

\theoremstyle{definition}
\newtheorem{definition}[thm]{Definition}
\newtheorem{remark}[thm]{Remark}
\newtheorem{example}[thm]{Example}

\numberwithin{equation}{section}

\begin{document}


\baselineskip=17pt



\title{Using boundaries to find smooth norms}
\date{\today}
\author{Victor Bible}
\address{Victor Bible, UCD School of Mathematical Sciences}
\email{victor.bible@ucdconnect.ie}


\begin{abstract}
The aim of this paper is to present a tool used to show that certain Banach spaces can be endowed with $C^k$ smooth equivalent norms. The hypothesis uses particular countable decompositions of certain subsets of $B_{X^*}$, namely boundaries. Of interest is that the main result unifies two quite well known results. In the final section, some new corollaries are given.
\end{abstract}

\thanks{The author is supported financially by Science Foundation Ireland under Grant Number `SFI 11/RFP.1/MTH/3112'}
\maketitle

\begin{section}{Introduction}
We say the norm of a Banach space $(X,||\cdot||)$ is $C^k$ smooth if its $k$th Fr\'{e}chet derivative exists and is continuous at every point of $X \setminus \{ 0 \}$. The norm is $C^{\infty}$ smooth if this holds for all $k \in \mathbb{N}$. This paper is concerned with the problem of establishing sufficient conditions for when a Banach space has a $C^k$ smooth renorming, for $k \in \mathbb{N} \cup \{ \infty \}$.
\begin{definition}A subset $B \subseteq B_{X^*}$ is a called a \emph{boundary} if for each $x$ in the unit sphere $S_X$, there exists $f \in B$ such that $f(x) = 1$.
\end{definition}

\begin{example}The following will be boundaries for any Banach space $X$.
\begin{enumerate} \label{eg3}
  \item The dual unit sphere $S_{X^*}$. This is a consequence of the Hahn-Banach Theorem.
  \item \label{eg3:2} The set of extreme points of the dual unit ball, Ext$(B_{X^*})$. This follows from the proof of the Krein-Milman Theorem (\cite[Fact 3.45]{fabetal}).
\end{enumerate}
\end{example}

Given $\varepsilon > 0$ and norms $|| \cdot ||$ and  $||| \cdot |||$ on a Banach space $X$, say $||| \cdot |||$ $\varepsilon$-\emph{approximates} $|| \cdot ||$ if, for all $x \in X$, $$(1-\varepsilon) || x || \leq |||x||| \leq (1 + \varepsilon) ||x||.$$ The notion of a boundary plays an important role in this area of study. Frequently, the existence of a boundary with certain properties gives rise to the desired renormings, as seen in the following result of H\'{a}jek, which is part of a more general theorem.

\begin{thm}[{\cite[Theorem 1]{hajek}}] \label{h}
If $(X,||\cdot||)$ admits a boundary contained in a $||\cdot||$-$\sigma$-compact subset of $B_{X^*}$, then $X$ admits an equivalent $C^{\infty}$ smooth norm that $\varepsilon$-approximates $||\cdot||$.
\end{thm}
H\'{a}jek and Haydon provided another sufficient condition for when this property holds, namely when $X = C(K)$ and $K$ is a compact Hausdorff $\sigma$-discrete space. We call a topological space $K$ \emph{$\sigma$-discrete} if $K = \bigcup_{n=0}^{\infty} D_n$, where each $D_n$ is relatively discrete: given $x \in D_n$, there exists $U_x$ open in $K$ such that $U_x \cap D_n = \{ x \}$.

\begin{thm}[{\cite[Theorem 5.1]{hajekhaydon}}] \label{hh}
Let $K$ be a $\sigma$-discrete compact space. Then, given $\varepsilon > 0$, $C(K)$ admits an equivalent $C^{\infty}$ smooth norm that $\varepsilon$-approximates $||\cdot||_{\infty}$.
\end{thm}
It is worth remarking that, in certain cases, these conclusions have been strengthened. In \cite{dfh}, it is shown that if $X$ has a countable boundary then $X$ has an equivalent analytic norm which $\varepsilon$-approximates the original norm. Moreover, if $C(K)$ admits an analytic renorming, then $K$ is countable \cite{h2}. For a norm $||\cdot||$ to be analytic we mean it is a real valued analytic function on $X \setminus \{ 0 \}$. Analytic functions on Banach spaces are defined and explored \cite{w}. 

The Orlicz functions $M$ for which the corresponding Orlicz sequence spaces $l_M$ and Orlicz function spaces $l_M(0,1),$ $l_M(0,\infty)$ have an equivalent $C^{\infty}$ smooth norm were characterised in \cite{mt}. Futhermore, the Orlicz sequence spaces $h_M$ with equivalent analytic norm were characterised in \cite{ht}.

The main result of this paper, Theorem \ref{main}, generalises these results as corollaries. It also takes into account smoothness of injective tensor products, in a manner similar to that of \cite{h}. As in the proof of \cite[Theorem 5.1]{hajekhaydon}, the proof of Theorem \ref{main} makes use of two lemmas (\cite[Lemma 5.2 and Lemma 5.3]{hajekhaydon}) concerning the so-called generalised Orlicz norm, denoted by $|| \cdot ||_{\phi}$. The first lemma provides a condition where $|| \cdot ||_{\phi}$ is equivalent to $|| \cdot ||$.

\begin{definition} \label{def1} Let $B$ be a set. Suppose for every element $t \in B$ there exists a convex function $\phi_t$ on $[0,\infty)$ with $\phi_t(0) = 0$ and $\lim_{\alpha \rightarrow \infty} \phi_t(\alpha) = \infty$ (such functions are called \emph{Orlicz functions}). Define $|| \cdot ||_{\phi}$ on $l_{\infty}(B)$ by
$$ ||f||_{\phi} = \text{ inf } \left \{ \rho > 0 : \sum_{t \in B} \phi_{t} \left(\frac{|f(t)|}{\rho} \right)  \leq 1 \right \}.$$ and define $\ell_{\phi} (B)$ as the set of $f \in \ell_{\infty}(B)$ satisfying $||f||_{\phi} < \infty$.
\end{definition}

\begin{lem} \cite[Lemma 5.2]{hajekhaydon} \label{lem1}
Let $||\cdot||_{\phi}$ be as in Definition \ref{def1}. Suppose there exist $\beta > \alpha > 0$ with the property $\phi_t(\alpha) = 0$ and $\phi_t(\beta) \geq 1$ for all $t \in B$. Then $\ell_{\phi}(B) \cong \ell_{\infty}(B)$ and
$$\alpha||\cdot||_{\phi} \leq ||\cdot||_{\infty} \leq \beta||\cdot||_{\phi}.$$
\end{lem}

We use $||\cdot||_{\phi}$ to define another norm on a more general space $X$, which we also denote by $||\cdot||_{\phi}$. The second lemma gives a sufficient condition for when $||\cdot||_{\phi}$ on $X$ is $C^k$ smooth. It uses the notion of local dependence on finitely many coordinates and generalises \cite[Lemma 5.3]{hajekhaydon}.

\begin{lem} \label{lem2}
Let $||\cdot||_{\phi}$ be as in Lemma \ref{lem1} and let $\Pi: X \rightarrow \ell_{\phi}(B)$ be an embedding (non-linear in general), where the map $x \mapsto \Pi(x)(t)$ is a seminorm which is $C^k$ smooth on the set where it is non-zero, for all $t \in B$. Assume the assignment $||x||_{\phi} = ||\Pi(x)||_{\phi}$ defines an equivalent norm on $X$. Suppose for each $x \in X$, with $||x||_{\phi} = 1$, there exists an open $U \subseteq X$ containing $x$, and finite $F \subseteq B$, such that $\phi_t(|y(t)|) = 0$ when $y \in U$ and $t \in B \backslash F.$ Finally, assume that each $\phi_t$ is $C^{\infty}$ smooth. Then $||\cdot||_{\phi}$ is $C^k$ smooth on $X$.
\end{lem}

As Lemma \ref{lem2} appears in \cite[Lemma 5.3]{hajekhaydon}, $X$ is taken to be a closed subspace of $\ell_{\infty}(B)$ and $\Pi$ is the identity. The proof uses the fact that each coordinate map $x \rightarrow |x(t)|$ is $C^{\infty}$ smooth on the set where it is non-zero and uses the implicit function theorem to show that $||\cdot||_{\phi}$ is also $C^{\infty}$ smooth. In our case, each coordinate map is $C^k$ smooth on the set where it is non-zero and the same argument guarantees that $||\cdot||_{\phi}$ is $C^k$ smooth.

The first part of the proof of Theorem \ref{main} is concerned with setting up the necessary framework to apply these lemmas. The remainder uses a series of claims to prove they do in fact hold. In the final section, Theorems \ref{h} and \ref{hh} are obtained as corollaries of Theorem \ref{main}, along with some other results and applications.

Before proceeding to the statement of Theorem \ref{main}, a key notion of $w^*$-\emph{locally relatively compact sets} ($w^*$-LRC for short) needs to be introduced. This property is first studied in \cite{pieces}, in the context of polyhedral norms.
\end{section}

\begin{definition}[{\cite[Definition 5]{pieces}}] Let $X$ be a Banach space. We call $E \subseteq X^*$ $w^*$-LRC if given $y \in E$, there exists a $w^*$-open set $U$ such that $y \in U$ and $\overline{E \cap U}^{||\cdot||}$ is norm compact.
\end{definition}

\begin{example}[{\cite[Example 6]{pieces}}] The following sets are $w^*$-LRC. \label{eg1}
\begin{enumerate} 
  \item Any norm compact or $w^*$-relatively discrete subset of a dual space.
  \item Given $X$ with an unconditional basis $(e_i)_{i \in I}$ and $f \in X^*$, define $$ \text{supp}(f) = \{ i \in I : f(e_i) \neq 0 \}.$$ \label{eg1-2} Let $E \subset X^*$ have the property that if $f, g \in E$, then $|\text{supp}(f)| = |\text{supp}(g)|  < \infty$. $E$ is $w^*$-LRC. Indeed, take $f \in E$ and define the $w^*$-open set $U = \{ g \in X^* : 0 < |g(e_i)| < |f(e_i)| + 1 : i \in \text{supp}(f) \}$. Clearly, if $g \in U \cap E$, then supp$(g)$ = supp$(f)$. Thus $U \cap E $ is a norm bounded subset of a finite dimensional space.
\end{enumerate}
\end{example}

\begin{remark} Evidently, $w^*$-LRC sets are preserved under scalar multiplication. Also, the family of $\sigma$-$w^*$-LRC subsets of a dual Banach space forms a $\sigma$-ideal. (This is because if $E$ is $w^*$-LRC and $F \subseteq E$, then $F$ is $w^*$-LRC. And of course any countable union of $\sigma$-$w^*$-LRC sets is again $\sigma$-$w^*$-LRC). But in general they do not behave well under straightforward linear and topological operations. To see this, consider the following.
\begin{enumerate}
\item Let $E = \{ e_n : n \in \mathbb{N} \}$ be the usual basis of $c_0$ and let $F = \{ 0 \}$. These sets are both $w^*$-LRC. However, $0$ is a $w^*$-accumulation point of $E$ and $||e_n - e_m||_{\infty} =1$ whenever $n \neq m$, so $E \cup F$ is not $w^*$-LRC.

\item The set $E = \{ \delta_{\alpha} + 2^{-n} \delta_{\alpha+n} : \alpha < \omega_1 \text{ is a limit ordinal}, n \in \mathbb{N} \}$ is $w^*$-discrete. But $\overline{E}^{||\cdot||} \supseteq  \{ \delta_{\alpha} : \alpha < \omega_1 \text{ is a limit ordinal} \}$. Using the fact that the ordinal $\omega_1$ is not $\sigma$-discrete, we can see that the set $\{ \delta_{\alpha} : \alpha < \omega_1 \text{ is a limit ordinal} \}$, and thus $\overline{E}^{||\cdot||}$, is not $\sigma$-$w^*$-LRC.

\item Consider the space $\ell_1 \oplus \ell_1 (B_{\ell_1}) \equiv \ell_1(\mathbb{N} \cup B_{\ell_1})$. Given $x \in B_{\ell_1}$, denote by $\overline{x}$ its canonical image in $\ell_1(\mathbb{N} \cup B_{\ell_1})$. Let $E = \{ \overline{x} \pm \delta_x : x \in B_{\ell_1} \}$. This set can be shown to be $w^*$-discrete but $E + E \supseteq \{ 2 \overline{x} : x \in B_{\ell_1} \} \equiv 2 B_{\ell_1}$. A conseqeunce of \cite[Proposition 12 (1)]{pieces} is that for an infinite dimensional space $X$, $S_{X^*}$ cannot be covered by a countable union of $w^*$-LRC sets. This result extends to $S_Y$, where $Y$ is any infinite-dimensional subspace of $X^*$. Because of this, $E + E$ is not $\sigma$-$w^*$-LRC.
\end{enumerate}
\end{remark}
The main result is concerned with renorming injective tensor products. Given Banach spaces $X$ and $Y$, the injective tensor product $X \otimes_{\varepsilon} Y$ is the completion of the algebraic tensor product $X \otimes Y$ with respect to the norm
$$||\sum_{i=1}^{\infty} x_i \otimes y_i|| = \sup \left \{ \sum_{i=1}^{\infty} f(x_i)g(y_i) : f \in B_{X^*}, g \in B_{Y^*} \right \} .$$

Also note the following facts. If $I_Y$ is the identity operator on $Y$, then given $f \in X^*$ we define $f^Y = f\otimes I_Y$ on $X \otimes Y$ by $f^Y(\sum_{i=1}^{\infty} x_i \otimes y_i) = \sum_{i=1}^{\infty} f(x_i)y_i$. We have $||f^Y|| =||f||$ and extend to the completion. Similarly define $g^X$ for $g \in Y^*$. A useful fact is $f \otimes g = g \circ f^Y = f \circ g^X. $

Given two boundaries $N \subseteq X^*$ and $M \subseteq Y^*$, the set $\{ f \otimes g : f \in N, g \in M \} $ is a boundary for $X \otimes_{\varepsilon} Y$. To see this, take $u \in X \otimes_{\varepsilon} Y$. There exists $f \in B_{X^*}$ and $g\in B_{Y^*}$ such that $||u|| = (f\otimes g)(u) = ||g^X (u)||$. Then there exists $\hat{f} \in N$ such that $\hat{f} (g^X (u))  = ||u|| = || \hat{f}^Y(u) ||$. Finally, there exists $\hat{g} \in M$ such that $\hat{g}(\hat{f}^Y(u)) = (\hat{f} \otimes \hat{g} )(u)  = ||u||.$

Given a Banach space $Y$ with a $C^k$ smooth renorming, Haydon gave a sufficient condition on $X$ for $X\otimes_{\varepsilon}Y$ to have a $C^k$ smooth renorming (\cite[Corollary 1]{h}). This condition involves a type of operator that are now known as Talagrand operators. Another sufficient condition is given in the main result below. It is worth noting that these conditions are incomparable. For example, the space $C[0,\omega_1]$ satisfies Haydon's condition but not that of Theorem \ref{main}. On the other hand, if we take $K$ to be the Ciesielski-Pol space as seen in \cite{cp}, then $C(K)$ satisfies the hypothesis of Theorem \ref{main} but not Haydon's condition.

\begin{section}{Main Result}
\begin{thm} \label{main} Let $X$ and $Y$ be Banach spaces and let $(E_{n})$ be a sequence of $w^{*}$-LRC subsets of $X^{*}$, such that $E = \bigcup_{n=0}^{\infty} E_{n}$ is $\sigma$-$w^{*}$-compact and contains a boundary of $X$. Suppose further that $Y$ has a $C^k$ smooth norm $|| \cdot ||_Y$ for some $k \in  \mathbb{N} \cup \{\infty\}$. Then $X \otimes_{\varepsilon} Y$ admits a $C^k$ smooth renorming that $\varepsilon$-approximates the canonical injective tensor norm.
\end{thm}

The proof of this theorem is based to some degree on that of \cite[Theorem 7]{pieces}. Given its technical nature, some of that proof is repeated here for clarity. We ask the reader to excuse any redundancy.

\begin{proof}
To begin, we can assume $E$ is a boundary and ${\overline{E_{n}}^{w^{*}}} \subseteq E$ for all $n  \in \mathbb{N}.$ Indeed, if neccessary, taking $E=\bigcup_{m=0}^{\infty} K_{m}$, where $K_m$ is $w^*$-compact, we can consider for all $n,m \in \mathbb{N},$
$$E_{n} \cap K_{m} \cap B_{X^{*}}.$$

By \cite[Proposition 12 (3)]{pieces} there exist $w^{*}$-open sets $V_{n}$ such that if we set $A_{n} = {\overline{E_{n}}^{w^{*}}} \cap V_{n}$, then
$$E_{n} \subseteq A_{n} \subseteq {\overline{E_{n}}^{|| \cdot ||}} \text{ and } A_{n} \text{ is } w^{*}\text{-LRC}.$$

Each $A_{n}$ is both norm $F_{\sigma}$ and norm $G_{\delta}$. So for each $n \in \mathbb{N}$, $A_{n} \backslash \bigcup_{k<n} A_{k}$ will in particular be norm $F_{\sigma}$. Now we write
$$A_{n} \backslash \bigcup_{k<n} A_{k} = \bigcup_{m=0}^{\infty} H_{n,m},$$

where each $H_{n,m}$ is norm closed. By arrangement, we assume $H_{n,m} \subseteq H_{n,m+1}$ for all $m \in \mathbb{N}$ and, for convenience, we set $H_{n,-1} = \varnothing $. Let $\pi : \mathbb{N}^{2} \longrightarrow \mathbb{N}$ be a bijection and for all $i,j\in \mathbb{N}$, define
$$L_{\pi(i,j)}=H_{i,j} \backslash H_{i,j-1}.$$

Clearly $E$ is the disjoint union of the $L_{n}$ and $\overline{L_{n}}^{w^{*}} \subseteq \overline{E_{p}}^{w^{*}} \subseteq E$, where $n=\pi (p,q).$ Given $f\in E,$ let
$$I(f) = \{ n \in \mathbb{N} : f \in \overline{L_{n}}^{w^{*}} \} \text{ and } n(f) = \text{min } I(f).$$

Now fix $\varepsilon > 0$. We define $\psi:E \longrightarrow (1,1+\varepsilon)$ by\
$$\psi(f) = 1 + \frac{1}{2} \varepsilon \cdot 2^{-n(f)} \left( 1+ \frac{1}{4} \sum_{i \in I(f)} 2^{-i} \right).$$

Set $\varepsilon_{n} = \frac{1}{96}\varepsilon \cdot 4^{-n}$. Fix $n$. As $\psi(L_{n}) \subseteq (1,1+\varepsilon),$ there is a finite partition of $L_{n}$ into sets J, such that diam$(\psi(J)) \leq \varepsilon_{n}.$

Let $P= \{I \subseteq J : I \text{ is } \varepsilon_{n} \text{-separated}\}.$ This set is non-empty because any singleton is in $P$. For a chain $T \subseteq P$ we have $ \bigcup_{N \in T} N \in P, $ so we can apply Zorn's Lemma to get $\varGamma \subseteq J,$ a maximal $\varepsilon_{n} \text{-separated}$ subset of $J.$ By maximality, $\varGamma$ is also an $\varepsilon_{n}$-net. And by the $\varepsilon_{n}$-separation, for a totally bounded set $M \subseteq J$, the intersection $M \cap \varGamma$ is finite.

By considering the finite union of these $\varGamma$, there exists $\varGamma_{n} \subseteq L_{n}$, with the property that given $f \in L_{n}$ there exists $h \in \varGamma_{n}$ so that
\begin{equation}\label{1}\tag{1} |\psi (f) - \psi (h) | \leq \varepsilon_{n} \text{ and } || f - h || \leq \varepsilon_{n}. \end{equation}

Moreover, if $M \subseteq L_{n}$ is totally bounded, $M \cap \varGamma_{n}$ is finite. Now define $B = \bigcup_{n=0}^{\infty} \varGamma_{n}$. We are now ready to define $|| \cdot ||_{\phi}$ on $\ell_{\infty}(B)$.

For each $f \in B$ we pick a $C^{\infty}$ Orlicz function $\phi_{f}$ so that
\begin{align*}
&\phi_{f}(\alpha) = 0 \text{ if } \alpha \leq \frac{1}{\psi(f)},\\
&\phi_{f}(\alpha) > 1 \text{ if } \alpha \geq \frac{1}{\theta(f)}, \text{ where } \theta(f) = \psi(f) - \varepsilon_{n}.
\end{align*}

We define $||\cdot||_{\phi}$ with respect to these functions, as per Definition \ref{def1}. By taking $(1 + \varepsilon)^{-1}$ and $1$ as the constants in the hypothesis of Lemma \ref{lem1} we have $l_{\phi}(B) \cong l_{\infty}(B)$ and $||\cdot ||_{\infty} \leq ||\cdot ||_{\phi} \leq (1 + \varepsilon) ||\cdot ||_{\infty}.$

We embed $X \otimes_{\varepsilon} Y$ into $\ell_{\infty} (B)$ by setting $\Pi(u) (f) = ||f^Y (u) ||_Y, f \in B$. The coordinate map $u \rightarrow ||f^Y (u)||$ is a seminorm which is $C^k$ smooth on the set where it is non-zero for each $f \in B$. Since $|| \Pi (u) ||_{\infty} = ||u||,$ it follows that $|| \cdot || \leq || \cdot ||_{\phi} \leq (1+ \varepsilon) ||\cdot||$ on $X$.

Suppose for the sake of contradiction that the remaining hypothesis of Lemma \ref{lem2} does not hold. Then we can find $u \in X \otimes_{\varepsilon} Y$ with $||u||_{\phi} = 1$, $(u_n) \subseteq X \otimes_{\varepsilon} Y$ with $u_n \rightarrow u$ and distinct $(f_n) \subseteq B$ such that $\phi_{f_n} (||f_{n}^{Y} (u_n)||) > 0$, for all $n$. Then $\psi(f_n) ||f_{n}^{Y} (u_n)|| > 1$ for all $n$.

Take a subsequence of $(f_n)$, again called $(f_n)$, such that $\psi(f_n) \rightarrow \alpha$ for some $\alpha \in \mathbb{R}$. Now take $(g_n) \subseteq S_{Y^*}$ such that $||f_{n}^{Y} (u_n)|| = g_n (f_{n}^{Y} (u_n)).$ Let $(f,g) \in B_{X^*} \times B_{Y^*}$ be an accumulation point of $(f_n,g_n)$ in the product of the $w^*$-topologies. Then $f \otimes g$ is a $w^*$-accumulation point of $(f_n\otimes g_n)$ and $\alpha (f \otimes g) (u) \geq 1$.

The remainder of the proof is concerned with obtaining the contradiction $\alpha (f\otimes g)(u) < 1.$

{\bf Case 1:} $\alpha = 1.$ With $\alpha = 1$, it is evident that $ \alpha(f\otimes g)(u) = (f\otimes g)(u)  \leq ||u||$. The following claim ensures $||u|| < 1$.\\
{\bf Claim 1:} If $v \neq  0,$ then $||v|| < ||v||_{\phi}.$ \\
Let $||v|| = 1$ and pick  $p \in E$, $q \in S_{Y^*}$ such that $1= (p \otimes q)(v).$ As noted above, this is possible because $E$ and $S_{Y^*}$ are boundaries of $X$ and $Y$, respectively. By (\ref{1}) above, let $r \in B$ such that $||p-r|| \leq \varepsilon_n$ for an appropriate $n$. Observe that $\theta(r)((r\otimes q) (v)) \leq ||v||_{\phi}$ holds. Indeed,
\begin{align*}
\sum_{l \in B} \phi_l \left ( \frac{||l^Y(v)||}{\theta(r) (r \otimes q)(v)} \right ) & \geq \phi_r \left ( \frac{||r^Y(v)||}{\theta(r) q(r^Y(v))} \right ) \\
& \geq \phi_r \left ( \frac{1}{\theta(r)} \right ) > 1. 
\end{align*}

Now to prove the claim,
\begin{align*}1 & = (p \otimes q) (v)\\ & = (r \otimes q) (v) + ((p-r) \otimes q) (v) \\ & = \theta(r) (r \otimes q) (v) + (1 - \theta (r)) (r \otimes q)(v) + ((p-r) \otimes q) (v) \\ & \leq ||v||_{\phi} + (1 - \theta (r)) (r \otimes q)(v) + ((p-r) \otimes q) (v).
\end{align*}
So we are done if $(\theta (r) -1) (r \otimes q)(v) + ((r-p) \otimes q) (v) > 0.$ Indeed,
\begin{align*}\theta(r) - 1 &= \psi(r) - \varepsilon_n - 1\\ &\geq \frac{1}{2} \varepsilon \cdot 2^{-n(r)} - \varepsilon_n\\ &\geq \frac{1}{2} \varepsilon \cdot 2^{-n} - \varepsilon_n.
\end{align*}

Also, $(r\otimes q)(v) = r(q^X (v)) \geq 1 - ||p-r||\cdot||q^X(v)|| \geq \frac{1}{2}.$ Thus,
\begin{align*}(\theta (r) -1) (r \otimes q)(v) + ((r-p) \otimes q) (v) & \geq \frac{1}{4} \varepsilon \cdot 2^{-n} - \frac{1}{2} \varepsilon_n - \varepsilon_n\\
&= \frac{1}{4} \varepsilon \cdot 2^{-n} - \frac{3}{2} \varepsilon_n\\
&= \frac{1}{4} \varepsilon \cdot 2^{-n} - \frac{1}{64} \varepsilon \cdot 4^{-n} > 0.
\end{align*}
And the claim is proven.\\

{\bf Case 2:} $\alpha > 1$.\\
We'll first prove $f \in E$.

Fix $N$ large enough so that $1 + \varepsilon \cdot 2^{-N} < \frac{1}{2} (1 + \alpha).$ Because $\psi(f_n) \rightarrow \alpha$ we have $\psi(f_m) > \frac{1}{2} (1 + \alpha)$ for all $m$ large enough. Hence, $n(f_m) < N$. Therefore, $f_m \in  \bigcup_{k<N} \overline{L_k}^{w*}$ for all such $m$. By $w^*$-closure, $f \in \bigcup_{k<N} \overline{L_k}^{w*} \subseteq E.$

Now the aim is to prove $\psi(f) > \alpha.$

We can assume $f_n \neq f \text{ for all } n \in \mathbb{N}$, because the $f_n$ are distinct.

Now fix the unique $m$ such that $f \in L_m$ and let
$$ J = I(f) \cup \{ k \in \mathbb{N} : k \geq m+2 \}.$$ Clearly $m \in I(f).$ Let $(p,q) \in \mathbb{N}^2$ such that $m = \pi (p,q).$ We have $L_m \subseteq A_p.$ Since $A_p$ is $w^*$-LRC, there exists a  $w^*$-open set $U \ni f$, such that $A_p \cap U$ is relatively norm compact.

From before, $\varGamma_{\pi (p,k)} \cap U$ is finite for all $k \in \mathbb{N}$, since $\varGamma_{\pi (p,k)} \subseteq A_p$. So the set
$$ V = U \backslash \left( \bigcup_{i \in \mathbb{N} \backslash J} \overline{L_i}^{w^*} \cup \left( \bigcup_{k=0}^q \varGamma_{\pi (p,k)} \backslash \{f \} \right) \right)$$ is $w^*$-open. Moreover, because $f \notin \bigcup_{i \in \mathbb{N} \backslash J} \overline{L_i}^{w^*},$ we have $f \in V.$ We assume from now on that $f_n \in V.$

{\bf Claim 2a} $m \notin I(f_n).$

If $m \in I(f_n)$, then
$$f_n \in \overline{L_m}^{w^*} \cap V \subseteq \overline{L_m \cap V}^{w^*} =  \overline{L_m \cap V}^{|| \cdot ||} \subseteq  \overline{L_m}^{|| \cdot ||} \subseteq H_{p,q}.$$

It follows that $f_n \in H_{p,k} \backslash H_{p,k-1} = L_{\pi(p,k)}$ for some $k \leq q.$ On the other hand, $f_n \in B,$ so $f_n \in L_{\pi(p,k)} \cap B = \varGamma_{\pi(p,k)}.$ However, this cannot be the case, since $f_n \in V \backslash \{ f \}.$\\

{\bf Claim 2b} $I(f_n) \subseteq J$.\\
Let $i \in I(f_n)$. If $i \notin J,$ then $f_n \in \bigcup_{j \in \mathbb{N} \backslash J}  \overline{L_j \cap V}^{w^*},$ but this contradicts $f_n \in V.$

{\bf Claim 2c} $\psi(f) - \psi(f_n) \geq \frac{1}{16} \varepsilon \cdot 4^{-m} = 6 \varepsilon_m$.\\
First note $n(f_n) \geq n(f),$ using Claim 2b and $n(f) = \text{ min } I(f) = \text{ min } J.$ There are two cases to consider. If $n(f_n) > n(f)$, then
$$\psi(f) - \psi(f_n) \geq 1 + \frac{1}{2} \varepsilon \cdot 2^{-n(f)} - (1 + \frac{3}{4} \varepsilon \cdot 2^{-n(f_n)}) \geq \frac{1}{8}  \varepsilon \cdot 2^{-n(f)} \geq \frac{1}{8} \varepsilon \cdot 2^{-m}.$$

And if $n(f_n) = n(f)$, then

\begin{align*}
\psi(f) - \psi(f_n) & \geq \frac{1}{8} \varepsilon \cdot 2^{-n(f)} \left( \sum_{i \in I(f)} 2^{-i} - \sum_{i \in I(f_n)} 2^{-i} \right)\\
& = \frac{1}{8} \varepsilon \cdot 2^{-n(f)} \left( \sum_{i \in I(f) \backslash I(f_n)} 2^{-i} - \sum_{i \in I(f_n) \backslash I(f)} 2^{-i} \right)\\
& \geq \frac{1}{8} \varepsilon \cdot 2^{-n(f)} \left( 2^{-m} - \sum_{i \in J \backslash I(f)} 2^{-i} \right) \\
& \geq \frac{1}{8} \varepsilon \cdot 2^{-n(f)} \cdot 2^{-m-1} \geq \frac{1}{16} \varepsilon \cdot 4^{-m}.
\end{align*}

{\bf Claim 2d} For $h \in B, ||h\otimes g||_{\phi} \leq \frac{1}{\theta(h)}.$

If $|(h \otimes g)(v)| > \frac{1}{\theta(h)},$ then $$\sum_{l \in B} \phi_{l}(||l^Y(v)||) \geq \phi_h (||h^Y(v)||) \geq \phi_h ((h \otimes g)(v))  > 1 \Longrightarrow ||v||_{\phi} > 1.$$ So, $||h\otimes g||_{\phi} = \sup \{ |(h \otimes g )(v)| : ||v||_{\phi} \leq 1 \} \leq \frac{1}{\theta(h)}.$ \\

We can now prove $\alpha (f\otimes g)(x) < 1.$
By \eqref{1}, take $h \in B$ such that $||f-h|| \leq \varepsilon_n$ and $|\psi(f) - \psi(h)| \leq \varepsilon_n$. We then have
\begin{align*} \alpha || f \otimes g ||_{\phi} &\leq \alpha (||h\otimes g||_{\phi} + ||(f-h)\otimes g||_{\phi})\\ 
&\leq \alpha(|| h \otimes g ||_{\phi} + || (f-h) \otimes g || )\\
& \leq \alpha \left ( \frac{1}{\theta(h)} + \varepsilon_n \right).
\end{align*}

So we are done if $\alpha(\frac{1}{\theta(h)} + \varepsilon_n) < 1.$
Well,
\begin{align*}& 1 - \frac{\alpha}{\theta(h)} - \alpha \varepsilon_n > 0 \\
\iff & \theta(h) - \alpha - \varepsilon_n \theta(h) \alpha> 0\\
\iff & \psi(h) - \varepsilon_n - \alpha - \varepsilon_n \theta(h) \alpha> 0.
\end{align*}

By claim 2c, we have $\psi(h) - \varepsilon_n - \alpha \geq 4 \varepsilon_n$ and since $\theta(h), \alpha < 2$, it follows that $ \varepsilon_n \theta(h) \alpha < 4 \varepsilon_n$.

And so, $\alpha ||f\otimes g ||_{\phi} < 1 \Longrightarrow \alpha (f\otimes g)(u) < 1.$
\end{proof}
\end{section}

\section{Applications}
\begin{cor} \label{cor1}
Suppose $X$ has a $\sigma$-$w^*$-LRC and $\sigma$-$w^*$-compact boundary. Then $X$ has a $C^{\infty}$ renorming.
\end{cor}

\begin{proof}
Apply Theorem \ref{main} to $X \otimes_{\varepsilon} \mathbb{R} = X$.
\end{proof}

We can now prove Theorems \ref{h} and \ref{hh} as corollaries of Corollary \ref{cor1}.

\begin{proof}[Proof of Theorem \ref{h}]
Any norm compact subset of $X^*$ is trivially $w^*$-LRC. The result follows from Corollary \ref{cor1}.
\end{proof}

\begin{proof}[Proof of Theorem \ref{hh}]
Let $K = \bigcup_{n=0}^{\infty} D_n$, where each $D_n$ is relatively discrete. Let $\delta_t$ be the usual evaluation functionals, $\delta_t(f) = f(t)$. Then $E_n = \{ \pm \delta_t : t \in D_n \}$ is $w^*$-relatively discrete and so $w^*$-LRC. Moreover, $E = \bigcup_{n=0}^{\infty} E_n$ is a $w^*$-compact boundary of $C(K)$ because given any $f \in C(K)$, there exists $t \in K$ such that $||f||_{\infty} = |f(t)|$, by compactness.
\end{proof}
The corollaries below are new results. Before presenting them, a definition and a theorem appearing in \cite{pieces} are needed.

\begin{definition}[{\cite[Definition 2]{pieces}}]
Let $X$ be a Banach space. We say a set $F \subseteq X^*$ is a \emph{relative boundary} if, whenever $x \in X$ satisfies $\sup \{ f(x) : f \in F \} = 1$, there exists $f \in F$ such that $f(x) = 1$.
\end{definition}

\begin{example} Any boundary and any $w^*$-compact set will be a relative boundary.

\end{example}

\begin{thm}[{\cite[Theorem 4]{pieces}}]\label{c}
Let $X$ be a Banach space and suppose we have sets $S_n \subseteq S_X$ and an increasing sequence $H_n \subseteq B_{X^*}$ of relative boundaries, such that $S_X = \bigcup_{n=0}^{\infty} S_n$ and the numbers
$$b_n = \inf \{ \sup \{h(x): h \in H_n\} : x \in S_n \} $$
are strictly positive and converge to 1. Then for a suitable sequence $(a_n)_{n=0}^{\infty}$ of numbers the set $F = \bigcup_{n=0}^{\infty} a_n (H_n \backslash H_{n-1})$ is a boundary of an equivalent norm.
\end{thm}

Given a Banach space with an unconditional basis $(e_i)_{i \in I}$ and $x = \sum_{i\in I} x_i e_i$, let $e_{i}^*(x) = x_i$. For $\sigma \subseteq I$, let $P_{\sigma}$ denote the projection given by $P_{\sigma} (x) = \sum_{i \in \sigma} e_{i}^{*}(x)e_i.$

\begin{cor} \label{b}
Let $X$ have a monotone unconditional basis $(e_i)_{i\in I}$, with associated projections $P_{\sigma}$, $\sigma \subseteq I$, and suppose we can write
 $S_X = \bigcup_{n=1}^{\infty} S_n$ in such a way that the numbers
   $$c_n = \inf \{ \sup \{ || P_{\sigma} (x) || : \sigma \subseteq I, |\sigma| = n \} : x \in S_n \} $$
are strictly positive and converge to 1. Then $X$ admits an equivalent $C^{\infty}$ smooth norm.
\end{cor}

\begin{proof}
Let $H_n = \{ h \in B_{X^*} :  |\text{supp}(h)| \leq n \} $. Each $H_n$ is a relative boundary because it is $w^*$-compact. Note that given $x \in S_n$ and $\sigma \subseteq I$, with $|\sigma| =n$,
\begin{align*}
  ||P_{\sigma}(x)|| & = \sup \{ f(P_{\sigma}(x)) : f \in B_{X^*} \}\\
  & = \sup \{ P_{\sigma}^* f(x) : f \in B_{X^*} \}.
\end{align*}

Of course, $|\text{supp}(P_{\sigma}^*f)| \leq n$, for all $f \in B_{X^*}$. And by monotonicity, $||P_{\sigma}^*||=1$. So $P_{\sigma}^*(f) \in H_n$. Therefore,
\begin{align*}
  0 < c_n & = \inf \{ \sup \{ || P_{\sigma} (x) || : \sigma \subseteq I, |\sigma| = n \} : x \in S_n \}\\
 & =  \inf \{ \sup \{  P_{\sigma}^* f (x) : f \in B_{X^*}, \sigma \subseteq I, |\sigma| = n\} : x \in S_n \}\\
 & = \inf \{ \sup \{ h(x) : h \in H_n \} : x \in S_n \} = b_n. 
\end{align*}
Thus, $(b_n)$ is a strictly positive sequence converging to 1.
%
The set $H_n \backslash H_{n-1}$ is $w^*$-LRC, by Example \ref{eg1}, (\ref{eg1-2}).

By Theorem \ref{c}, there exists a sequence $(a_n)_{n=0}^{\infty}$, where the set $F = \bigcup_{n=0}^{\infty} a_n (H_n \backslash H_{n-1})$ is a $\sigma$-$w^*$-LRC and $\sigma$-$w^*$-compact boundary for an equivalent norm $|||\cdot|||$. By Corollary \ref{cor1}, $X$ will admit an equivalent $C^{\infty}$-smooth that $\varepsilon$-approximates $|||\cdot|||$.
\end{proof}

\begin{cor}\label{cor14}
Let $X$ be a Banach space with a monotone unconditional basis $(e_i)_{i \in I}$ and suppose for each $x \in S_X$ there exists $  \sigma \subset I, |\sigma| < \infty$, so that $||P_{\sigma} (x)|| = 1$. Then $X$ admits an equivalent $C^{\infty}$-smooth norm that $\varepsilon$-approximates the original norm.
\end{cor}

\begin{proof}
Let $H_n = \{ h \in B_{X^*} :  |\text{supp}(h)| \leq n \} $. As mentioned in the proof of Corollary \ref{b}, each $H_n$ is $w^*$-compact and the finite union of $w^*$-LRC sets. Now take $x \in S_X$ and $\sigma$ such that $||P_{\sigma}(x)|| = 1$. Then there is $f \in B_{X^*}$ such that $$1 = ||P_{\sigma}(x)|| = f(P_{\sigma}(x)) = P_{\sigma}^* f(x).$$ Because $(e_i)_{i \in I}$ is monotone, $||P_{\sigma}^*|| = 1$ and so $P_{\sigma}^* f \in H_{|\sigma|}$. Therefore, the set $H = \bigcup_{n=0}^{\infty} H_n$ is a boundary satisfying the hypothesis of Corollary \ref{cor1}.
\end{proof}

Using Corollary \ref{b} we can obtain new examples of spaces with equivalent $C^{\infty}$ smooth renormings.

\begin{example}
Let $\mathbb{N} = \bigcup_{n=0}^{\infty} A_n$, where each $A_n$ is finite, and let $p = (p_n)$ be an unbounded increasing sequence of real numbers with $p_n \geq 1$.
For each sequence of real numbers $x = (x_n)$ define
$$\Phi(x) = \sup \left\{ \sum_{n=0}^{\infty} \sum_{k \in B_n} |x(k)|^{p_n} : B_n \subset A_n \text{ and } B_n \text{ are pairwise disjoint.} \right\}$$
\end{example}

\begin{proof}
We define $\ell_{A,p}$ as the space of sequences $x$ where $\Phi(x/\lambda) < \infty$ for some $\lambda > 0$, with norm $||x|| = \inf \{ \lambda > 0 : \Phi(x/\lambda) \leq 1 \}.$
Define the subspace $h_{A,p}$ as the norm closure of the linear space generated by the basis $e_n(k) = \delta_{n,k}$. \cite[Example 16]{pieces} provides an appropriate sequence of subsets $(S_n)$ of $S_X$ so that Corollary \ref{b} holds.
\end{proof}

\begin{example}
Let $M$ be an Orlicz function with
$$ M(t) > 0 \text{ for all  } t > 0, \text{ and } \lim_{t \rightarrow 0} \frac{M(K(t)}{M(t)} = + \infty,$$
for some constant $K > 0$. Let $h_M(\Gamma)$ be the space of all real functions $x$ defined on $\Gamma$ with $\sum_{\gamma \in \Gamma} M (x_{\gamma} / \rho ) < \infty$ for all $\rho > 0$, with the norm
$$ ||x|| = \inf \left \{ \rho > 0 : \sum_{\gamma \in \Gamma} M \left( \frac{x_{\gamma}}{\rho} \right) \leq 1 \right \}.$$
\end{example}

\begin{proof}
The canonical unit vector basis $(e_{\gamma})_{\gamma \in \Gamma}$ of functions $e_{\gamma} (\beta) = \delta_{\gamma, \beta}$ is unconditionally monotone. \cite[Example 18]{pieces} provides suitable subsets of $S_X$ to ensure the hypothesis of Corollary \ref{b} holds.
\end{proof}

The final example concerns the predual of a Lorentz sequence space $d(w, 1, A)$, for an arbitrary set $A$.

Let $w = (w_n) \in c_0 \backslash \ell_1$ with each $w_n$ strictly positive and $w_0 = 1$. We define $d(w, 1, A)$ as the space of $x: A \longrightarrow \mathbb{R}$ for which
$$||x|| = \sup \left \{ \sum_{j=0}^{\infty} w_j |x(a_j)| :  (a_j)  \subseteq A  \text{ is a sequence of distinct points }\right \} < \infty. $$ The canonical predual $d_*(w, 1, A)$ of $d(w, 1, A)$ is given by the space of $y: A \longrightarrow \mathbb{R}$ for which $\overline{y} = (\overline{y}_k) \in c_0$, where
$$\overline{y}_k = \sup \left \{ \frac{\sum_{i=0}^{k-1} |y(a_i)|}{\sum_{i=0}^{k-1} w_i} : a_0, a_1, \dots, a_{k-1} \text{ are distinct points of } A \right \}, $$
with norm $||y|| = ||\overline{y}||_{\infty}$. We can see that $(e_a)_{a \in A}$ is a monotone unconditional basis for both $d(w, 1, A)$ and $d_*(w, 1, A)$. The separable version of $d_*(w, 1, A)$ was first introduced in \cite{s}.

\begin{example} \label{lor}
$X = d_*(w,1,A)$ has a $C^{\infty}$ smooth equivalent renorming that $\varepsilon$-approximates the original norm.
\end{example}

\begin{proof}
Let $y \in S_X$. Since $\overline{y} \in c_0$, there exists $k \in \mathbb{N}$ such that $\overline{y}^k = 1.$ It can also be shown $y \in c_0(A)$ and thus the supremum in the definition of $\overline{y}^k$ is attained. Following this, there exists $a_0, a_1, \dots , a_{k-1} \in A$ such that $$1 = \overline{y}^k = \frac{\sum_{i=0}^{k-1} |y(a_i)|}{\sum_{i=0}^{k-1} w_i}.$$ Setting $\sigma = \{a_0, a_1, ... , a_{k-1}\}$, we have $||P_{\sigma}(y)|| =1$. By Corollary \ref{cor14}, $X$ has a $C^{\infty}$ smooth equivalent renorming that $\varepsilon$-approximates the original norm.

\end{proof}

\begin{remark}
The space $X = d_*(w,1,A)$ for $A$ uncountable is a new example of a space with a $C^{\infty}$ smooth renorming. It is not yet known if $X$ has an analytic renorming.
\end{remark}

\begin{remark} In Theorem \ref{main} and Corollary \ref{cor1} we cannot drop the $\sigma$-$w^*$-compactness condition in general, and expect an equivalent norm of any order of smoothness that depends locally on finitely many coordinates. In \cite{fhz}, $C_0(\omega_1)$ is shown to have no such norm. On the other hand, $C_0(\omega_1)$ admits an equivalent norm supporting a boundary that is $w^*$-discrete (this follows from \cite[Theorem 10]{fpst}).
\end{remark}

\section{Acknowledgements} The author would like to thank R. J. Smith for discussion and suggestions throughout the writing of this paper and S. Troyanski for further remarks, in particular bringing his attention to Example \ref{lor}.

\end{document}